\documentclass[12pt]{amsart}

\usepackage{amssymb}
\usepackage{amsmath}
\usepackage{amsthm}
\usepackage{enumerate}
\usepackage{quoting}

\usepackage{xcolor}

\theoremstyle{plain}

\newtheorem{theorem}{Theorem}[section]

\newtheorem{lemma}[theorem]{Lemma}
\newtheorem{proposition}[theorem]{Proposition}
\newtheorem{corollary}[theorem]{Corollary}

\newtheorem{observation}[theorem]{Observation}

\theoremstyle{definition}

\newtheorem{definition}[theorem]{Definition}
\newtheorem{remark}[theorem]{Remark}
\newtheorem{example}[theorem]{Example}

\newcommand{\R}{\mathbb{R}}
\newcommand{\N}{\mathbb{N}}
\newcommand{\Z}{\mathbb{Z}}

\renewcommand{\epsilon}{\varepsilon}
\renewcommand{\phi}{\varphi}

\begin{document}

\title{\large Extendability of continuous quasiconvex functions from subspaces}

\author{Carlo Alberto De Bernardi and Libor Vesel\'y}

\address{Dipartimento di Matematica per le Scienze economiche, finanziarie ed attuariali, Universit\`{a} Cattolica del Sacro Cuore, Via Necchi 9, 20123 Milano, Italy}

\address{Dipartimento di Matematica\\
	Universit\`a degli Studi\\
	Via C.~Saldini 50\\
	20133 Milano\\
	Italy}

\email{carloalberto.debernardi@unicatt.it}
\email{carloalberto.debernardi@gmail.com}
\email{libor.vesely@unimi.it}

 \subjclass[2010]{Primary 26B25, 46A55; Secondary 52A41, 52A99}

 \keywords{quasiconvex function, extension, topological vector space}

 \thanks{}

\begin{abstract}
Let $Y$ be a subspace of a topological vector space $X$,  and $A\subset X$ an open convex set that intersects $Y$. 
We say that the property $(QE)$ [property $(CE)$] holds if every continuous quasiconvex [continuous convex] function on $A\cap Y$ admits a
continuous quasiconvex [continuous convex] extension defined on $A$.
We study relations between $(QE)$  and $(CE)$ properties, proving that $(QE)$ always implies $(CE)$ and that, under suitable hypotheses (satisfied for example if $X$ is a normed space and $Y$ is a closed subspace of $X$), the two properties are equivalent. 

By combining the previous implications between $(QE)$  and $(CE)$ properties with known results about the property $(CE)$, we obtain some new positive results about the extension of quasiconvex continuous functions. In particular, we generalize the results contained in \cite{DEQEX} to the infinite-dimensional separable case. Moreover, we also immediately obtain existence of examples in which $(QE)$ does not hold.    
\end{abstract}

\maketitle

%%%%%%%%%%%%%%%%%%%%%%%%%%%%%%%%%%%%%%%%%%%%%%%%%%%%%%

%%%%%%%%%%%%%%%%%% Headings
\markboth{C.A.\ De Bernardi and L.\ Vesel\'y}{Extendability of continuous quasiconvex functions}
%%%%%%%%%%%%%%%%%%%%%%%%%%%

\section{Introduction}\label{S:intro}

A real-valued function $f$, defined on a convex set (in a vector space), is said to be {\em quasiconvex} if all sub-level sets
of $f$ are convex (see Section~\ref{sec: notation}). Quasiconvex functions represent a natural generalization
of convex functions, and
play an important role in Mathematical programming, in Mathematical economics,  
and in many other areas of Mathematical analysis  
(see \cite{AH,BFitzV,gencon,CR,CRdiff, pierskalla,penot} and the references therein for results  and  applications concerning quasiconvex functions).

The present paper deals with an extension problem for continuous quasiconvex functions. For simplicity, let us introduce the
following {\em temporary terminology.} 

%\smallskip

\begin{quoting}
\noindent
Let $Y$ be a subspace of a topological vector space $X$,  and $A\subset X$ an open convex set that intersects $Y$. 
Let us say that $(QE)$ [$(CE)$] holds if every continuous quasiconvex [continuous convex] function on $A\cap Y$ admits a
continuous quasiconvex [continuous convex] extension defined on $A$.
\end{quoting}
%
%\smallskip
%
Notice that the extension properties $(QE)$ and $(CE)$ seem a priori to be two independent problems. 
As we shall see, this is not completely the case.

In the last two decades, the property $(CE)$ has been studied in several
papers, mainly for the particular case of $A=X$ (see e.g. \cite{BV2002,BMV,VZ2010,DEVEX1,DEVEX2,DEX} and the references therein).
The known positive results require assumptions of a separability nature.

As for the quasiconvex case, in the recent paper \cite{DEQEX}, the first-named author 
proved that $(QE)$ holds in some particular, essentially finite-dimensional cases 
(namely, for either $X$ a finite-dimensional normed space, or $X$ an arbitrary normed space,
$Y$ finite-dimensional and $A\cap Y$ a polytope in $Y$). 
However, the proof of the basic result in
\cite{DEQEX} relies on a compactness argument, which cannot be used in the infinite-dimensional case.

In the present paper, we study the property $(QE)$ for general, possibly infinite-dimensional topological vector spaces, obtaining
some new positive results. With respect to \cite{DEQEX}, our approach here is different, consisting in studying
implication relations between $(QE)$ and $(CE)$.

After some preparatory results in Section~\ref{sec: notation}, the following Section~\ref{S:2ext_prop} provides
 the main tool of the present paper:
a characterization of the property $(QE)$,
in terms of ``extendability'' of certain increasing families of open (in $Y$) convex
subsets of $A\cap Y$ to increasing families of open convex subsets of $A$, preserving an appropriate topological property.
On the other hand, to characterize the property $(CE)$, it suffices to consider a similar property for increasing {\em sequences}
of sets, as proved in \cite{DEVEX1} (see also Theorem~\ref{all}). The main results of the paper are contained in  Section~\ref{S:relationship}, in which we use the two characterizations cited above
to study relationship between
$(QE)$ and $(CE)$. First we show that if $(QE)$ holds then $(CE)$ holds as well and, under a weak additional assumption on $X$, 
the subspace $Y$ is closed (see Theorem~\ref{QEimpliesCE}).
Then we prove that the vice-versa holds
whenever there exists a nonnegative continuous convex function on $A$ which is null exactly 
on $A\cap Y$ (see Theorem~\ref{CEimpliesQE}). Notice that this condition is automatically satisfied in the case when
$Y$ is a closed subspace of a normed space $X$.

Section~\ref{S:appl} starts with a simple proposition showing that, under an additional assumption, if $(QE)$ holds then it
is possible to extend 
a non-constant $f$ so that it preserves the range and the set of minimizers of $f$. In the rest of the section, 
 results from previous sections and \cite{DEVEX1, DEVEX2} are applied to obtain
sufficient conditions on $X,Y$ assuring that $(QE)$ holds for $A=X$ (Theorem~\ref{applCE}), as well as for an arbitrary open 
convex set $A\subset X$ (Theorem~\ref{applSCE}). An example of our results for Banach spaces is the following.
\smallskip

\noindent
{\em Let $Y$ be a closed subspace of a Banach space $X$. If $X/Y$ is separable then $(QE)$ holds for $A=X$. If $Y$ is separable 
and either
$Y$ is complemented or $X$ is WCG (weakly compactly generated; e.g., reflexive or separable) then $(QE)$ holds
for arbitrary open convex set $A\subset X$. Moreover, in both cases, the extension preserves the set of minimizers.}
\smallskip

\noindent
Moreover, at the end of Section~\ref{S:appl} we present some examples  in which $(QE)$ does not hold.
Also in this case we apply some results concerning $(CE)$ from \cite{DEVEX1}.   

Finally, let us remark that some positive results on extendability of quasiconvex functions 
from convex sets (instead of subspaces) are contained in our forthcoming paper
\cite{DEVEX_UC} while examples of non-extendability are given in \cite{DEVEqc_examples}. In both cases, an 
important role is played by geometric properties of the convex set in question.

%%%%%%%%%%%%%%%%%%%%%%%

\medskip

\section{Preliminaries}\label{sec: notation}

We consider only nontrivial real topological vector spaces ({\em t.v.s.}, for short). Unless specified otherwise,
such spaces are not assumed to be Hausdorff.
If $X$ is a t.v.s.\ and $E\subset Z\subset X$, we denote by $\overline{E}^Z$, $\partial_Z E$, and $\mathrm{int}_Z E$ 
the relative closure, boundary and interior of $E$ in $Z$.

All functions we consider are real-valued. Let us recall that, if $A$ is a convex subset of $X$, a function $f:A\to\R$ is called {\em quasiconvex} if 
$f((1-t)x+t y)\le \max\{f(x),f(y)\}$ whenever $x,y\in A$ and $t\in[0,1]$. Thus $f$ is quasiconvex if and only if
all its strict sub-level sets
$$
[f< \beta]:=\{x\in A:\, f(x)< \beta\}\qquad(\beta\in\R)
$$
are convex, if and only if all its sub-level sets 
$$
[f\le\beta]:=\{x\in A:\, f(x)\le\beta\}
\qquad(\beta\in\R)
$$ 
are convex. The corresponding super-level sets $[f>\beta]$ and $[f\ge\beta]$ are defined analogously.

Given sets $A_n$ ($n\in\N$) and $A$, the notation $A_n\nearrow A$ means that $A_n\subset A_{n+1}$ ($n\in\N$) and $\bigcup_{n\in\N}A_n=A$;
and $A_n\searrow A$ means that $A_n\supset A_{n+1}$ ($n\in\N$) and $\bigcap_{n\in\N}A_n=A$.

%If $X$ is a normed linear space then $X^*$ is its dual Banach space, and
% $U_X$, $B_X$, and $S_X$ are the open unit ball, the closed unit ball, and the unit sphere of $X$, respectively.
% Other notation is standard, and various topological notions refer to the norm topology of $X$, if not specified otherwise. In the case the topological notions considered refer to a subset $Z$ of $X$, we will specify it in our notation; i.e., if $E$ is a subset of $Z$, we will denote by 
%  Moreover, given $D,H\subset X$, we say that $D$ is dense in $H$ if $H\subset \overline D$.

 For $x,y\in X$, $[x,y]$ denotes the closed segment in $X$ with endpoints $x$ and $y$, and $[x,y):=[x,y]\setminus\{y\}$; meanings of
$(x,y]$ and $(x,y)$ are analogous.

For a convex set $C\subset X$,
it is well known (cf.\ \cite{Klee}) that: if  
$x\in\mathrm{int}_X C$ and $y\in \overline{C}^X$ then $[x,y)\subset\mathrm{int}_X C$; and if $C$ has nonempty interior then
%$\overline{\inte C}=\overline C$ and
 $\mathrm{int}_X(\overline{C}^X)=\mathrm{int}_X C$.

\begin{lemma}\label{closure}
Let $Y$ be a subspace of a t.v.s.\ $X$, and $A\subset X$ an open convex set intersecting $Y$. Then:
\begin{enumerate}[(a)]
\item $A\cap\overline{Y}^X\subset\overline{A\cap Y}^X$;
\item $\overline{A\cap Y}^X=\overline{A}^X\cap\overline{Y}^X$;
\item $A\cap Y$ is relatively closed in $A$ if and only if $Y$ is closed (in $X$).
\end{enumerate} 
\end{lemma}

\begin{proof}
(a) follows easily since $A$ is open.

(b) The inclusion $\subset$ is obvious. To show the other inclusion, take any
$x\in \overline{A}^X\cap\overline{Y}^X$, and fix some $y\in A\cap Y$.
Then $[y,x)\in \mathrm{int}_X(\overline{A}^X)\cap\overline{Y}^X=A\cap\overline{Y}^X
\subset\overline{A\cap Y}^X$ by (a). It follows that $x\in\overline{A\cap Y}^X$.

(c) We can (and do) assume that $0\in A$. Assume that $A\cap Y$ is relatively closed in $A$. Then by (b),
$A\cap Y=\overline{A\cap Y}^X\cap A=\overline{A}^X\cap\overline{Y}^X\cap A=A\cap \overline{Y}^X$.
Consequently, $Y=\bigcup_{t>0}t(A\cap Y)=\bigcup_{t>0}t(A\cap \overline{Y}^X)=\overline{Y}^X$,
and so $Y$ is closed. The other implication is trivial.
\end{proof}

The following lemma coincides with \cite[Lemma~1.1]{DEVEX1}.

\begin{lemma}\label{conv}
Let $Y$ be a subspace of a t.v.s.\ $X$,
$C\subset Y$ and $A\subset X$ convex sets. Then:
\begin{enumerate}
\item[(a)] $\mathrm{conv}(A\cup C)\cap Y=\mathrm{conv}[(Y\cap A)\cup C]$;
\item[(b)] if $C$ is open in $Y$, $A$ is open in $X$ and
$A\cap C\ne\emptyset$, then $\mathrm{conv}(A\cup C)$ is open in $X$.
%\item[(c)] If $\mathrm{int}\,A\ne\emptyset$, and $A$ is dense in
%an open set $H\subset X$,
%then $A\supset H$.
%\item[(d)]
%Let $C$ be open in $Y$ and $A$ open in $X$. If
%$Y$ is dense in $X$ and $C\subset A$, then there exists an
%open convex set $B\subset A$
%such that $\emptyset\ne B\cap Y\subset C$.
\end{enumerate}
\end{lemma}

One of the main tools of the present paper will be a simple proposition, 
Proposition~\ref{P:lowerlevelsets}, on constructing
quasiconvex functions from increasing families of open convex sets. For brevity of formulations, we shall
use the following terminology.

\begin{definition}\label{D:Omega}
Let $Z$ be a topological space, and $E\subset Z$. By an {\em $\Omega(E)$-family in $Z$} we mean a family
$\{D_\alpha\}_{\alpha\in \R}$ of sets in $Z$ such that 
\begin{align*}
\bigcap_{\alpha\in\R}D_\alpha=\emptyset,\ 
\bigcup_{\alpha\in\R}D_\alpha=E,\  
\text{and}\ 
\overline{D}^E_\alpha\subset D_\beta \text{ whenever $\alpha,\beta\in\R$, $\alpha<\beta$.}
\end{align*}
\end{definition}

The second part of the following proposition, that is, reconstructing a quasiconvex function from its sublevel sets,
is trivial and part of the mathematical folklore, and also the first part appears implicitly or in a slightly different form elsewhere (e.g., in \cite{DEQEX,CR}).

\begin{proposition}\label{P:lowerlevelsets} 
Let $Z$ be a convex set in a t.v.s.\ $X$.

\smallskip

\begin{enumerate}[(a)]
\item If $\{D_\alpha\}_{\alpha\in\R}$ is an $\Omega(Z)$-family of relatively open convex subsets of $Z$,
then the function
$$
f(x):=\sup\{\alpha\in\R:\,x\notin D_\alpha\} \qquad (x\in Z)
$$
is a continuous quasiconvex function on $Z$ such that for each $\alpha\in\R$
\begin{equation}\label{lls}
[f<\alpha]=\bigcup_{\beta<\alpha}D_\beta
\quad\text{and}\quad
[f\le\alpha]=\bigcap_{\gamma>\alpha}D_\gamma\,.
\end{equation}

\smallskip

\item If $g$ is a continuous quasiconvex function on $Z$, then the sets
$$
D_\alpha:=[g<\alpha]\qquad(\alpha\in\R)
$$
form an $\Omega(Z)$-family of relatively open convex subsets of $Z$. Moreover, the function
$f$ associated to $\{D_\alpha\}_{\alpha\in\R}$ by (a) coincides with $g$.
\end{enumerate}
\end{proposition}

\begin{proof}
\ 

(a) It is clear that $f$ is well-defined, and it is an elementary exercise to show that it satisfies \eqref{lls}.
Since all the sets $[f<\alpha]$ are convex and relatively open in $Z$, $f$ is quasiconvex and upper semicontiuous on $Z$.
Moreover, the definition of an $\Omega(Z)$-family implies that
$[f\le\alpha]=\bigcap_{\gamma>\alpha}\overline{D}^Z_\gamma$ for each $\alpha\in\R$, and hence
$f$ is also lower semicontinuous on $Z$.

(b) The first part is clear. For the second part, if $x\in Z$ then we have $f(x)=\sup\{\alpha\in\R:\, g(x)\ge\alpha\}=g(x)$.
\end{proof}

%%%%%%%%%%%%%%%%
\begin{example}
Consider the family $\{D_\alpha\}_{\alpha\in\R}$ of open convex sets in $Z=\R$ given by: $D_\alpha=\emptyset$ for $\alpha<0$,
$D_\alpha=(-\infty,0)$ for $0\le\alpha<1$, and $D_\alpha=\R$ for $\alpha\ge1$. Then the family $\{D_\alpha\}_{\alpha\in\R}$
is nondecreasing but it is not an $\Omega(\R)$-family. The corresponding function $f$ defined in Proposition~\ref{P:lowerlevelsets}(a) 
is the characteristic function of the interval $[1,\infty)$, and hence it is quasiconvex but not lower semicontinuous.
\end{example}
%%%%%%%%%%%%%%%%

Let $Y$ be a closed subspace of  a t.v.s. $X$. Let us recall that the {\em quotient-space topology} on $X/Y$ is the vector
topology whose open sets are the sets of the form $q(A)$, $A\subset X$ open, 
where $q\colon X\to X/Y$ is the quotient map. Moreover, $X/Y$ is locally convex whenever $X$ is (see, e.g., \cite{KN}). We shall also say
that a set $E\subset X$ is a {\em convexly $G_\delta$-set} if $E$ is the intersection of a sequence of open convex sets.
Also recall that a convex set $C$ in a t.v.s.\  is called {\em algebraically bounded} if it contains no half-lines.

Since the proof of the following technical lemma is a bit long, we postpone it to Appendix.

\begin{lemma}\label{L:G-delta}
Let $X$ be a t.v.s., $Y\subset X$ a subspace, and $A\subset X$ an open convex set intersecting $Y$. 
Consider the following three groups of conditions.

\smallskip

\begin{enumerate}
\item[A.]
\begin{enumerate}
\item[($a_1$)] There exists a continuous quasiconvex function $\delta\colon A\to[0,\infty)$ such that
$\delta^{-1}(0)=A\cap Y$;
\item[($a_2$)] $A\cap Y$ is a convexly $G_\delta$-set in $X$.
\end{enumerate}

\medskip

\item[B.]
\begin{enumerate}
\item[($b_1$)] There exists a continuous quasiconvex function $f\colon X/Y\to[0,\infty)$ such that
$f^{-1}(0)=\{0\}$;
\item[($b_2$)] the origin of $X/Y$ is a convexly $G_\delta$-set in $X/Y$.
\end{enumerate}

\medskip

\item[C.]
\begin{enumerate}
\item[($c_1$)] There exists a continuous convex function $d\colon X\to[0,\infty)$ such that $d^{-1}(0)=Y$;
\item[($c_2$)] there exists a continuous seminorm $p$ on $X$ such that $p^{-1}(0)=Y$;
\item[($c_3$)] there exists an open convex set $V\subset X$ such that $\bigcap_{\epsilon>0}\epsilon V=Y$;
\item[($c_4$)] $X/Y$ contains an algebraically bounded convex neighborhood of $0\in X/Y$;
\item[($c_5$)] there exists a continuous norm on $X/Y$.
\end{enumerate}
\end{enumerate}

\smallskip
\noindent
Then the conditions of any single group are mutually equivalent and, for
$i,j\in\{1,2\}$ and $k\in\{1,\dots,5\}$, one has the implications
$$
(a_i)\Leftarrow(b_j)\Leftarrow(c_k).
$$
Moreover, if $A=X$ then also $(a_i) \Rightarrow (b_j)$.
\end{lemma}

\begin{proof}
See Appendix.
\end{proof}

\medskip

%%%%%%%%%%%%%%%%%%%%%%%%%%%%%%%%%%%%%%%%%%%%%%%%%%%%%%

\section{Two extension properties and their characterizations}\label{S:2ext_prop}

\begin{definition}
Let $Y$ be a subspace of a t.v.s.\ $X$, and $A\subset X$ an open convex set that intersects $Y$.
We shall say that the couple $(X,Y)$ has:
\begin{enumerate}[(a)]
\item the {\em $CE(A)$-property} if each continuous convex function $f\colon A\cap Y\to\R$ can be extended to a continuous convex
function $F\colon A\to\R$;
\item the {\em $QE(A)$-property} if each continuous quasiconvex function $f\colon A\cap Y\to\R$ can be extended to 
a continuous quasiconvex function $F\colon A\to\R$.
\end{enumerate}
\end{definition}

The following characterization of the $CE(A)$-property was proved in \cite[Theorem~2.6]{DEVEX1}.

\begin{theorem}\label{all}
Let $X$ be a t.v.s., $Y\subset X$ a subspace, and $A\subset X$ an open convex set intersecting $Y$. Then the following assertions are equivalent.
\begin{enumerate}
\item[(i)] Every continuous convex function on $A\cap Y$ admits a continuous convex extension to $A$.
\item[(ii)] For every sequence $\{C_n\}$ of open convex sets in $Y$ such that $C_n\nearrow A\cap Y$, there exists a sequence $\{D_n\}$ of open convex sets in $X$ such that $D_n\nearrow A$ and $D_n\cap Y=C_n$, for each $n\in\N$.
\item[(ii')] For every sequence $\{C_n\}$ of open convex sets in $Y$ such that $C_n\nearrow A\cap Y$, there exists a sequence $\{D_n\}$ of open convex sets in $X$ such that $D_n\nearrow A$ and $D_n\cap Y\subset C_n$, for each $n\in\N$.
\end{enumerate}
\end{theorem}

The aim of the present short section is to provide a characterization of the $QE(A)$-property in a similar spirit. In this case, sequences are not enough; we have to consider
monotone families of convex sets defined 
in Definition~\ref{D:Omega}. The main tool is Proposition~\ref{P:lowerlevelsets}.

\begin{theorem}\label{teo:CharExQuasiconv}
Let $X$ be a t.v.s., $Y\subset X$ a subspace, and $A\subset X$ an open convex set intersecting $Y$. Then the following assertions are equivalent.
	\begin{enumerate}
		\item $(X,Y)$ has the
		$\mathrm{QE}(A)$-property.
		\item For each $\Omega(A\cap Y)$-family $\{C_\alpha\}_{\alpha\in\R}$ of open convex sets in $Y$ 
		there exists an $\Omega(A)$-family $\{D_\alpha\}_{\alpha\in\R}$ of open convex sets in $X$ such that
		\begin{equation}\label{eqA}
		\bigcup_{\gamma<\alpha}C_\gamma \subset D_\alpha\cap Y\subset C_\alpha \qquad(\alpha\in\R).
		\end{equation}	
	\end{enumerate}
\end{theorem}

\begin{proof} $(\mathrm i)\Rightarrow(\mathrm{ii})$.
Let  $\{C_\alpha\}_{\alpha\in\R}$ be an $\Omega(A\cap Y)$-family of open convex sets in $Y$. 
By Proposition~\ref{P:lowerlevelsets}(a), the function $f:A\cap Y\to \R$, $f(y):=\sup\{\alpha\in\R:  y\not\in C_\alpha\}$, is a continuous quasiconvex function
on $A\cap Y$ such that $\bigcup_{\gamma<\alpha}C_\gamma=[f<\alpha]$.
Let $F\colon A\to\R$ be a continuous quasiconvex extension of $f$. For $\alpha\in\R$, define $D_\alpha:=[F<\alpha]$, and observe that
$\overline{D}^A_\alpha\subset D_\beta$ whenever $\alpha<\beta$ are real numbers. Thus $\{D_\alpha\}_{\alpha\in\R}$
is an $\Omega(A)$-family of open convex sets in $X$. Moreover, 
$\bigcup_{\gamma<\alpha}C_\gamma=[f<\alpha]=D_\alpha\cap Y\subset C_\alpha$ for each $\alpha\in\R$.

$(\mathrm{ii})\Rightarrow(\mathrm{i})$. Let $f\colon A\cap Y\to \R$ be a continuous quasiconvex function. For each $\alpha\in\R$ define 
$C_\alpha:=[f<\alpha]$, and notice that $\{C_\alpha\}_{\alpha\in\R}$
is an $\Omega(A\cap Y)$-family of open convex sets in $Y$. 
Let $\{D_\alpha\}_{\alpha\in\R}$ be an $\Omega(A)$-family of open convex sets in $X$ that satisfies \eqref{eqA}. 
By Proposition~\ref{P:lowerlevelsets}, the function $F\colon A\to \R$, given by $F(x):=\sup\{\alpha\in\R :  x\not\in D_\alpha \}$, 
is continuous and quasiconvex on $A$. Moreover, Proposition~\ref{P:lowerlevelsets}(b) gives that 
for every $y\in A\cap Y$ we have $F(y)=\sup\{\alpha\in\R : y\notin C_\alpha\}=f(y)$.
Hence $F$ extends $f$, and we are done.
\end{proof}

\medskip

%%%%%%%%%%%%%%%%%%%%%%%%%%%%%%%%%%%%%%%%%%%%%%%%%%%%

\section{Relationship between the properties $QE(A)$ and $CE(A)$}\label{S:relationship}

The aim of the present section is to show that the quasiconvex extension property $QE(A)$ implies the
convex extension property $CE(A)$, and that the two properties are equivalent under some additional
assumptions. We shall also see that the $QE(A)$-property is possible only for closed subspaces $Y$. Notice that
this is not the same for the $CE(A)$-property; indeed, by \cite[Lemma~3.2]{DEVEX1} and its proof,
{\em $(X,Y)$ has the $CE(A)$-property if and only $(X,\overline{Y}^X)$ has the $CE(A)$-property.}

\smallskip

\begin{theorem}\label{QEimpliesCE}
Let $X$ be a t.v.s., $Y\subset X$ a subspace, and $A\subset X$ an open convex set that intersects $Y$.
Assume that the couple $(X,Y)$ has the $QE(A)$-property. Then $(X,Y)$ has the $CE(A)$-property.

If, moreover, the singleton $\{0\}\subset X$ is a convexly $G_\delta$-set (e.g., if $X$ is locally convex and metrizable)
then $Y$ is closed.
\end{theorem}

\begin{proof}
Assume that $(X,Y)$ has the $QE(A)$-property. 
To show that $(X,Y)$ has the $CE(A)$-property, it suffices to verify condition (ii') in Theorem~\ref{all}.
Let $\{C_n\}_{n\in\N}$ be a sequence of open convex sets in $Y$ such that $C_n\nearrow A\cap Y$. We can (and do)
assume that $0\in C_1$. Let us define a family $\{B_\alpha\}_{\alpha\in\R}$ as follows.

For $n\in\N$ put $C_n':=\frac{n}{n+1} C_n$,  and notice that these sets satisfy:
\begin{itemize}
	\item $C_n'\nearrow A\cap Y$;
	\item for each $n\in\N$, there exists $\lambda_{n+1}\in(0,1)$ such that  $C'_n\subset\lambda_{n+1}C'_{n+1}$.
\end{itemize}
Put $\lambda_1:=1/2$ and, for each $n\in\N$, consider the increasing affine function $\phi_n:\R\to\R$ defined by $$\phi_n(t)=t-n+(n+1-t)\lambda_{n},\qquad\qquad t\in\R.$$ 
Notice that  $\phi_n(n)=\lambda_{n}$ and $\phi_n(n+1)=1$.
Now, given a real number $\alpha\ge1$, take the unique $n\in\N$ with $\alpha\in[n,n+1)$, and define
$
B_\alpha:=\phi_n(\alpha)\,C'_{n}$. For $\alpha<1$, define $B_\alpha:=\emptyset$. Observe that $B_\alpha\subset C_n'\subset\lambda_{n+1}C_{n+1}'=B_{n+1}$ for $\alpha\in[n,n+1)$ and $n\in\N$.

It is easy to verify that $\{B_\alpha\}_{\alpha\in\R}$ is an $\Omega(A\cap Y)$-family of open convex sets in $Y$.
By Theorem~\ref{teo:CharExQuasiconv}, there exists an $\Omega(A)$-family $\{D_\alpha\}_{\alpha\in\R}$ of open convex sets in $X$
that satisfies \eqref{eqA}. In particular, $D_n\nearrow A$, and
$D_n\cap Y\subset B_{n}\subset C_n$ for each $n\in\N$.
By Theorem~\ref{all}, $(X,Y)$ has the $CE(A)$-property.

Now, let the origin of $X$ be a convexly $G_\delta$-set.
If $Y$ is not closed, by Lemma~\ref{closure}(c) there exists $\bar{x}\in\overline{A\cap Y}^A\setminus Y$. Clearly
$\bar{x}\in A$. By Lemma~\ref{L:G-delta} there exists a continuous quasiconvex function $\delta\ge0$ on $X$ such that $\delta^{-1}(0)=0$.
Then the formula $f(y):=\log\delta(y-\bar{x})$ defines a continuous quasiconvex function on $A\cap Y$ that admits no
continuous extension defined at $\bar{x}$. But this contradicts the $QE(A)$-property.
\end{proof}

The remaining part of the present section is devoted to the implication ``$\,CE(A)\Rightarrow QE(A)\,$''.
Let us start with the following technical lemma.

\begin{lemma}\label{L:cl}
Let $X$ be a t.v.s., $Y\subset X$ a closed subspace, and 
$A\subset X$ an open convex set intersecting $Y$. Let $d\colon A\to[0,\infty)$
be a continuous convex function such that $d^{-1}(0)=A\cap Y$.
Let $U, D_2$ be open convex subsets of $A$, and $C$ an open convex set in $Y$.
Assume also that:
\begin{enumerate}[(a)]
\item $U\cap Y\subset C$ and $\overline{C}^{A\cap Y}\subset D_2$;
\item there exists a neighborhood $W\subset X$ of the origin such that $U+W\subset D_2$;
\item $d(U)$ is bounded.
\end{enumerate}
Then the set $D_1:=\mathrm{conv}(U\cup C)$ satisfies
$\overline{D}_1^A\subset D_2$.
\end{lemma}

\begin{proof}
%We claim that $\overline{D}_1^X\subset D_2\cup Y$. 
Fix an arbitrary $x\in \overline{D}_1^A$. Then $x$ is the limit of a net $(z_i)_{i\in I}$ in $D_1$, 
and hence we can write
$$
x=\lim_{i\in I} \bigl[ t_i u_i +(1-t_i)c_i \bigr]
\quad\text{with}\quad t_i\in[0,1],\ u_i\in U,\ c_i\in C.
$$
Passing to a subnet if necessary,
we can (and do) assume that $t_i\to\bar{t}\in[0,1]$. 

If $\bar{t}=0$ then by the properties of $d$ we have 
$$
d(x)\le\liminf_{i}[t_i d(u_i)+(1-t_i) d(c_i)]=\liminf_{i}t_i d(u_i)=0,
$$
and hence $x\in A\cap Y$. 
By Lemma~\ref{conv}(a) and Lemma~\ref{closure},
$x\in \overline{D}_1^X\cap Y\cap A=\overline{D_1\cap Y}^X\cap A=\overline{C}^Y\cap A\subset D_2$. 

If $\bar{t}>0$, fix $\eta>0$ such that $\eta/\bar{t}<1$, and then choose
$i\in I$ so that $x\in  t_i u_i +(1-t_i)c_i+\eta W$ and $\eta/t_i<1$. Then
\begin{align*}
x &\textstyle \in 
t_i U +\eta W +(1-t_i)C = t_i\bigl(U+\frac{\eta}{t_i}W\bigr)+(1-t_i)C \\
&\subset t_i(U+W)+(1-t_i)C\subset D_2.
\end{align*}
This completes the proof.
\end{proof}

%%%

\begin{observation}\label{O:LC-like}
Let $Y$ be a subspace of a t.v.s.\ $X$, and $A\subset X$ an open convex set intersecting $Y$.
Let $(X,Y)$ have the $CE(A)$-property. Then for each open convex set $C\subset A\cap Y$ in $Y$
there exists an open convex set $H\subset A$ in $X$ such that $H\cap Y=C$.

\smallskip\noindent
\rm
(Indeed, it suffices to apply Theorem~\ref{all}(ii) with $C_1:=C$ and $C_n:=A\cap Y$ for any integer $n\ge2$.)
\end{observation}

\begin{theorem}\label{CEimpliesQE}
Let $Y$ be a closed subspace of a t.v.s.\ $X$, and $A\subset X$ an open convex set that intersects $Y$. 
Assume there exists a continuous convex function $d\colon A\to[0,\infty)$ such that $d^{-1}(0)=A\cap Y$.
%\emph{(For instance, this is satisfied whenever $X/Y$ admits a continuous norm.)}
If the couple $(X,Y)$ has the $CE(A)$-property then it has the $QE(A)$-property as well.
\end{theorem}

\begin{proof}
Let $(X,Y)$ have the $CE(A)$-property, and let $\{C_\alpha\}_{\alpha\in\R}$ be an $\Omega(A\cap Y)$-family of 
open convex sets in $Y$.
The set $J:=\{\alpha\in\R : C_\alpha\ne\emptyset\}$ is clearly an upper-unbounded interval. If $J=[\alpha_0,\infty)$,
we can (and do) assume that $\alpha_0=1$. Otherwise, $J$ is an open interval and we can (and do) assume that $J=\R$.
In both cases, $C_1$ is nonempty, and hence we can (and do) assume that $0\in C_1$. 

Since $C_n\nearrow A\cap Y$ as $n\in\N$ tends to infinity, by Theorem~\ref{all} there exists a sequence
$\{E_n\}_{n\in\N}$ of open convex sets in $X$ such that $E_n\nearrow A$, and $E_n\cap Y=C_n$ for each $n\in\N$.
Define
$K_n:=\frac{n}{n+1}\bigl(E_n\cap[d<n]\bigr)$ ($n\in\N$), and notice that these sets satisfy:
$$
K_n\nearrow A,\quad K_n\subset [d<n],\quad\text{and}\quad
K_n\subset\lambda_{n+1}K_{n+1}\quad\text{with}\quad \lambda_{n+1}\in(0,1).
$$
For every $\alpha\ge1$ take the unique $n\in\N$ with $\alpha\in[n,n+1)$, and define
$\phi_n(\alpha):=\alpha-n+(n+1-\alpha)\lambda_{n}$,
$$
U_\alpha:=\phi_n(\alpha) K_{n}
\quad\text{and}\quad
D_\alpha:=\mathrm{conv}(U_\alpha\cup C_\alpha).
$$
Notice that $\phi_n(n)=\lambda_n$, $\phi_n(n+1)=1$, and $\phi_n$ is
an increasing affine function. 
Since $U_\alpha\cap Y=\phi_n(\alpha)\frac{n}{n+1}C_n\subset C_n\subset C_\alpha$, we have $D_\alpha\cap Y= C_\alpha$
by Lemma~\ref{conv}. Moreover, since $U_\alpha\subset K_n$, $d$ is bounded on $U_\alpha$.

It is easy to see that for real numbers $1\le\alpha<\beta$ we always have
$0\in U_\alpha \subset\lambda_{\alpha,\beta}U_\beta$
for some $\lambda_{\alpha,\beta}\in(0,1)$.

In the case when $J=[1,\infty)$, we easily complete the proof by defining $D_\alpha=\emptyset$ for $\alpha<0$;
indeed, an application of Lemma~\ref{L:cl} shows that then the family $\{D_\alpha\}_{\alpha\in\R}$
has the  properties from Theorem~\ref{teo:CharExQuasiconv}(ii).

Now consider the case $J=\R$. Let us construct open convex sets $D_\alpha$ ($\alpha<1$) proceeding by induction
with respect to the (unique) integer $n\le 0$ such that $\alpha\in[n,n+1)$.
Assume that, for some integer $n\le0$, we have already constructed $D_\beta$ for all $\beta\ge n+1$. 
(For $n=0$, this is the case.) By Observation~\ref{O:LC-like}, there exists an open convex set $H_n\subset A$
such that $H_n\cap Y=C_n$. Define $E_n:=H_n\cap D_{n+1}\cap[d<\frac1{|n|+2}]$, and notice that
$E_n\cap Y\subset C_n$.
Fix an arbitrary $y_n\in E_n\cap Y$, and consider the open convex set $W_n:=(1/2)(E_n-y_n)$ containing $0$.
Then we have $y_n+2W_n=E_n$. For $\alpha\in[n,n+1)$ we define
$$
U_\alpha:=(y_n+W_n) + (\alpha-n)W_n
\quad\text{and}\quad
D_\alpha:=\mathrm{conv}(U_\alpha\cup C_\alpha).
$$ 
Since $U_\alpha\cap Y\subset E_n\cap Y\subset C_n\subset C_\alpha$, we have
$D_\alpha\cap Y=C_\alpha$ by Lemma~\ref{conv}. Notice also that
$U_\alpha+(n+1-\alpha)W_n=E_n\subset D_{n+1}\cap[d<\frac1{|n|+2}]$ for each such $\alpha$.

So we have defined a family $\{D_\alpha\}_{\alpha\in\R}$ of open convex sets in $X$.
An application of Lemma~\ref{L:cl} allows us to conclude that
$\overline{D}_\alpha^A\subset D_{\beta}$ for any couple of real numbers $\alpha<\beta$.
Moreover, $\bigcap_{\alpha\in\R}D_\alpha=\bigcap_{n\in\N}D_{-n}\subset d^{-1}(0)=A\cap Y$ and hence
$\bigcap_{\alpha\in\R}D_\alpha=\bigcap_{\alpha\in\R}D_\alpha\cap Y=\bigcap_{\alpha\in\R}C_\alpha=\emptyset$.
Now, it easily follows that 
$\{D_\alpha\}_{\alpha\in\R}$ is an $\Omega(A)$-family such that $D_\alpha\cap Y=C_\alpha$ for each $\alpha\in\R$.  
In particular, it satisfies the condition (ii) in Theorem~\ref{teo:CharExQuasiconv}, as needed.

\end{proof}

%%%

\medskip

%%%%%%%%%%%%%%%%%%%%%%%%%%%%%%%%%%%%%%%%%%%%%%%%%%%%%

\section{Some applications}\label{S:appl}

Let us start with a simple result on extensions of continuous quasiconvex functions that preserve the range and the
set of minimizers.

\begin{proposition}\label{better}
Let $X$ be a t.v.s., $Y\subset X$ a closed subspace, and $A\subset X$ an open convex set that intersects $Y$.
Assume also that:
\begin{enumerate}[(a)]
\item the couple $(X,Y)$ has the $QE(A)$-property;
\item there exists a continuous quasiconvex function $\delta\colon A\to[0,+\infty)$ such that $\delta^{-1}(0)=A\cap Y$.
\end{enumerate}
Then each non-constant continuous quasiconvex function $f\colon A\cap Y\to\R$ admits a continuous quasiconvex extension
$F\colon A\to \R$ such that $F(A)=f(A\cap Y)$ and,
for $a:=\inf F(A)$, 
$$
\{ x\in A : F(x)=a\}=\{ y\in A\cap Y : f(x)=a\}.
$$
\end{proposition}

\begin{proof}
Given $f$, denote $a:=\inf f(A\cap Y)$ and $b:=\sup f(A \cap Y)$, and notice that $a<b$ since $f$ is not constant.
If $f<b$ on $A\cap Y$, fix an increasing homeomorphism $\phi$ of $(-\infty,b)$ onto $\R$. Then $g:=\phi\circ f$ is a 
(non-constant) continuous quasiconvex function on $A\cap B$. By (a), $g$ admits a continuous quasiconvex extension
$G\colon A\to\R$. Then $F_1:=\phi^{-1}\circ G$ is a continuous quasiconvex extension of $f$ such that $F_1<b$.
If $b\in f(A\cap Y)$, take an arbitrary continuous quasiconvex extension $G\colon A\to \R$ of $f$, and define
$F_1(x):=\min\{G(x),b\}$ ($x\in A$). Then $F_1$ is a continuous quasiconvex extension of $f$ such that $F_1\le b$.

If $a=-\infty$ we simply put $F:=F_1$, and we are done. Now, assume that $a\in\R$. Fix some $\epsilon>0$ such that $a+\epsilon(\pi/2)<b$,
and define $F(x):=\max\{a+\epsilon\arctan{\delta(x)}, F_1(x)\}$ ($x\in A$). Then $F$ is clearly continuous and quasiconvex,
and $F(x)=F_1(x)$ whenever $F_1(x)>a+\epsilon(\pi/2)$. Moreover,
for $x\in A$, one has
$F(x)=a$ if and only $x\in Y$ and $f(x)=F_1(x)=a$. It follows that $F$ has the desired properties.
\end{proof}

\begin{remark}\label{constant}
Notice that, under the assumptions of the above theorem, if $f\equiv c$ is constant on $A\cap Y$ then 
$F(x):=c+\delta(x)$ ($x\in A$) is a quasiconvex extension of $f$ with the same set of minimizers (but, of course, not the same
range).
\end{remark}

Thanks to Theorem~\ref{CEimpliesQE}, we can now obtain some sufficient conditions for extendability of continuous quasiconvex functions.
Let us state some examples. A t.v.s.\ is called {\em conditionally separable} (see \cite{DEVEX1}) if for each neighborhood $V$ of $0$
there exists a countable set $E$ such that $E+V=X$. For metrizable spaces, conditional separability is equivalent to separability.
In general, separability implies conditional separability but not vice-versa.

\begin{theorem}[Extension to the whole space]\label{applCE}
Let $Y$ be a closed subspace of a t.v.s.\ $X$. 
Assume that at least one of the following conditions holds.
\begin{enumerate}[(a)]
\item $Y$ is complemented in $X$.
\item $X$ is locally convex,  and $X/Y$ is conditionally separable and admits a continuous norm.
\end{enumerate}
Then each continuous quasiconvex function $f$ on $Y$ admits a continuous quasiconvex extension $F$ defined on
the whole $X$.

Moreover, if $f$ is non-constant and, in the case (a) also the origin of $X/Y$ is a countable intersection of open convex sets, 
then the extension $F$ can be chosen so that\ \ 
$F(X)=f(Y)$\ \  and\ \  $\mathrm{arg\,min}_X F= \mathrm{arg\,min}_{Y} f\,$. 
\end{theorem}

\begin{proof}
(a) Let $P\colon X\to Y$ be a continuous linear projection onto $Y$. Then clearly $F:=f\circ P$ is 
a continuous quasiconvex extension of $f$. So $(X,Y)$ has the $QE(X)$-property. The second part
follows from Lemma~\ref{L:G-delta} and Proposition~\ref{better}.

(b) By \cite[Theorem~4.5]{DEVEX1}, $(X,Y)$ has the $CE(X)$-property. By Lemma~\ref{L:G-delta} and
Theorem~\ref{CEimpliesQE}, $(X,Y)$ has the $QE(X)$-property. The last part follows again from Proposition~\ref{better}.
\end{proof}

Let us recall that a Banach space $X$ is said to have the {\em separable complementation property}
if every separable subspace of $X$ is contained in a separable complemented subspace of $X$.
It is known (see e.g.\ \cite[pp.\ 481--482]{PY}) that this property holds in any of the following cases:
(i) $X$ is weakly compactly generated;
(ii) $X$ is dual and has the Radon-Nikod\'ym property;
(iii) $X$ has a countably norming M-basis;
(iv) $X$ is a Banach lattice not containing $c_0$.

\begin{theorem}[Extension to an open convex set]\label{applSCE}
Let $A$ be an open convex set in a locally convex t.v.s.\ $X$, and $Y\subset X$ 
a closed subspace such that $X/Y$ admits a continuous norm.
Assume that $A\cap Y\ne\emptyset$ and at least one of the following conditions is satisfied.
\begin{enumerate}[(a)]
\item $Y$ is finite-dimensional.
\item $X$ is conditionally separable.
\item $Y$ is conditionally separable and complemented.
\item $X$ is a Banach space with the separable complementation property and $Y$ is separable.
\item $A\cap Y$ is weakly open in $Y$, and either $Y$ is complemented or $X/Y$ is conditionally separable.
\item $Y$ is a Banach space, $A\cap Y$ is a countable intersection of open half-spaces in $Y$, and 
either $Y$ is complemented or $X/Y$ is conditionally separable.
\end{enumerate}
Then each
continuous quasiconvex function $f$ on $A\cap Y$ can be extended to a continuous quasiconvex function $F$
on $A$.

Moreover, if $f$ is non-constant then the extension $F$ can be chosen so that\ \ 
$F(A)=f(A\cap Y)$\ \  and\ \  $\mathrm{arg\,min}_A F= \mathrm{arg\,min}_{A\cap Y} f\,$. 
\end{theorem}

\begin{proof} Since $X/Y$ admits a continuous norm, equivalence of the conditions in Lemma~\ref{L:G-delta}, C, ensures existence of a continuous seminorm on $X$ whose kernel is $Y$.
By Theorem~\ref{CEimpliesQE} it suffices to show that $(X,Y)$ has the $CE(A)$-property. This is true
for (a) by \cite[Proposition~3.7(b)]{DEVEX1}. The cases (b),(c) and (d) 
were proved in \cite[Corollary~2.3]{DEVEX2}. 
If either $Y$ is complemented or $X/Y$ is conditionally separable then $(X,Y)$ has the $CE(X)$-property
by Proposition~3.3(a) and Theorem~4.5 in \cite{DEVEX1}. Now, (e) and (f) follow from 
Corollary~2.4(a) and Corollary~2.6(a) in \cite{DEVEX2}, respectively. 
Finally, the last part follows from Proposition~\ref{better}.
\end{proof}

\medskip

Let us finally present some examples  in which $QE(X)$ does not hold.
 As a consequence of Theorem~\ref{QEimpliesCE}, \cite[Proposition~4.4]{DEVEX1}, \cite[Proposition~4.5]{DEVEX1}, \cite[Corollary~4.6]{DEVEX1}, and \cite[Remark~5.6]{DEVEX1}, we immediately obtain the following corollary.

\begin{corollary}
	The couple $(X,Y)$ fails the ${QE}(X)$-property in any of
	the following cases.
	\begin{enumerate}
		\item $Y$ is an infinite-dimensional
		closed subspace of $X=\ell_\infty(\Gamma)$ isomorphic to some $c_0(\Lambda)$ or
		$\ell_p(\Lambda)$ with $1<p<\infty$.
		\item $Y$ is an infinite-dimensional
		closed subspace of $X=\ell_\infty(\Gamma)$ not containing any isomorphic copy of $\ell_1$.
		\item   $X$ is a Grothendieck Banach space but this is not the case for $Y$.
		\item $Y$ is a separable nonreflexive Banach space, considered as a subspace of $X=\ell_\infty$.
		\item $X=L^p([0,1])$ with $0<p<1$ and  $Y$ is a finite-dimensional subspace of $X$.
	\end{enumerate}
	In particular, $(\ell_\infty,c_0)$ fails the ${QE}(X)$-property.
\end{corollary}

%%%%%%%%%%%%%%%%%%%%%%%%%%%%%%%%%%%%

\medskip

\section*{Appendix: proof of Lemma~\ref{L:G-delta}}

\begin{proof}[Proof of Lemma~\ref{L:G-delta}]
 \ 

A. The implication
$(a_1)\Rightarrow(a_2)$ follows immediately by considering the open convex sets $U_n:=[\delta<1/n]$ ($n\in\N$).
Let us show that
$(a_2)\Rightarrow(a_1)$. We can (and do) assume that $0\in A$.
Let $\{U_k\}_{k\in\N}$ be a decreasing sequence of open 
convex sets in $A$ such that $\bigcap_{k\in\N}U_k=A\cap Y$.
Define $V_k:=\frac{k+1}{k}U_k$ ($k\in\N$), and notice that 
$$
\text{$\textstyle V_{k+1}\subset\frac{k(k+2)}{(k+1)^2}V_k$ 
($k\in\N$), and $\textstyle A\cap Y\subset \bigcap_{k\in\N}V_k\subset 2\bigcap_{k\in\N}U_k=2A\cap Y$.}
$$
Now let us define
a family $\{W_n\}_{n\in\Z}$ of open convex sets in $X$ by
$$
W_n:=V_{|n|}\ \text{for $n<0$ integer,}\ \text{and}\ 
W_n:=(n+2)V_1\ \text{for $n\ge0$ integer.}
$$
It is easy to see that, for each $n\in\Z$,
$$
W_n\subset\lambda_{n+1}W_{n+1}\quad\text{with some $\lambda_{n+1}\in(0,1)$.}
$$
Let us extend $\{W_n\}_{n\in\Z}$ to a family $\{W_\alpha\}_{\alpha\in\R}$ by defining
$$
W_\alpha:=\phi_n(\alpha)W_n\quad\text{whenever $\alpha\in[n,n+1)$, $n\in\Z$,}
$$
where $\phi_n(\alpha):=(n+1-\alpha)\lambda_n+(\alpha-n)$.
Since each $\phi_n$ is an increasing affine function such that $\phi_n(n)=\lambda_n$ and $\phi_n(n+1)=1$, 
it is easy to see that for each couple of reals $\alpha<\beta$ we have
$W_\alpha\subset\lambda_{\alpha,\beta}W_\beta$ with $\lambda_{\alpha,\beta}\in(0,1)$, and hence
$\overline{W}^{X}_\alpha\subset W_\beta$. Moreover,
$A\cap Y\subset\bigcap_{\alpha\in\R}W_\alpha\subset 2A\cap Y$ and $\bigcup_{\alpha\in\R}W_\alpha=X$.

Fix an increasing homeomorphism $h$ of $(0,\infty)$ onto $\R$, and define
$$
D_\alpha:=W_{h(\alpha)}\cap A\ \text{for $\alpha>0$,}
\quad\text{and}\quad
D_\alpha:=\emptyset\ \text{for $\alpha\le0$.}
$$
Then $\{D_\alpha\}_{\alpha\in\R}$ is an $\Omega(A)$-family of open convex sets in $X$.  
Proposition~\ref{P:lowerlevelsets} assures that the formula 
$$
\delta(x):=\sup\{\alpha\in\R; \, x\notin D_\alpha\}\quad(x\in A)
$$
defines a continuous quasiconvex real-valued function on $A$. Clearly, $\delta\ge0$ and by Proposition~\ref{P:lowerlevelsets},
$$
\delta^{-1}(0)=[\delta\le 0]=\bigcap_{\alpha>0}D_\alpha=\bigcap_{n\in\N}V_n\cap A=A\cap Y.
$$

\smallskip

B. The equivalence $(b_1)\Leftrightarrow(b_2)$ follows by the first part of the proof
(applied with $Y=\{0\}\subset X/Y$).

\smallskip

C. Conditions $(c_k)$.

$(c_4)\Rightarrow(c_5)$. If ($c_4$) holds, there exists $U\subset X/Y$, 
an algebraically bounded, symmetric, open convex neighborhood of the origin in $X/Y$.
The Minkowski gauge of $U$ is a continuous seminorm on $X/Y$, which is even a norm since
$U$ is algebraically bounded.

$(c_5)\Rightarrow(c_2)$. If $\nu$ is a continuous norm on $X/Y$, then $p:= \nu\circ q$ (where $q$ is the quotient map)
satisfies ($c_2$).

$(c_2)\Rightarrow(c_1)$ is obvious.

$(c_1)\Rightarrow(c_3)$. If $d$ is as in ($c_1$), let us define $V:=[d<1]$. If $y\in Y$ then 
$y/\epsilon\in V$ ($\epsilon>0$), and hence $y\in\bigcap_{\epsilon>0}\epsilon V$. On the other hand,
if $x\in\bigcap_{\epsilon>0}\epsilon V$ then $d(x/\epsilon)<1$ for each $\epsilon>0$. By convexity,
for each $\epsilon\in(0,1)$ we obtain $d(x)=d\bigl(\epsilon(x/\epsilon)+(1-\epsilon)0\bigr)\le \epsilon d(x/\epsilon)<\epsilon$.
Then $d(x)=0$, that is, $x\in Y$. Thus $V$ satisfies ($c_3$).

$(c_3)\Rightarrow(c_4)$. Let $V$ be as in ($c_3$). It is easy to see that $V+Y=V$ since $V$ is convex and open. 
It follows that
the open convex neighborhood $U:=q(V)$ of $0\in X/Y$ satisfies $\bigcap_{\epsilon>0}\epsilon U=\{0\}$.
This implies that $U$ is algebraically bounded. We are done.

\smallskip

D. The implication $(c_5)\Rightarrow(b_2)$ is obvious. If $(b_2)$ holds then the function $\delta:=f\circ q$ 
(where $q\colon X\to X/Y$ is the quotient map) satisfies the condition contained in $(a_1)$.

\smallskip

E. Now, for $A=X$, let us show that $(a_2)\Rightarrow(b_1)$. 
Let $U_n\subset X$ ($n\in\N$) be open convex sets such that $\bigcap_{n\in\N}U_n=Y$.
For each $n\in\N$,  $U_n+Y=U_n$ since $U_n$ is convex and open. It follows easily that the intersection of the open convex sets 
$V_n:=q(U_n)\subset X/Y$ ($n\in\N$) is just the origin of $X/Y$.

\end{proof}

\bigskip

\subsection*{Acknowledgement}

The research of the first author has been partially
supported by the GNAMPA (INdAM -- Istituto Nazionale di Alta Matematica) Research Project 2020 and by the Ministry for Science and Innovation, Spanish State Research Agency (Spain),  under project 
PID2020-112491GB-I00.
The research of the second author has been partially supported by the GNAMPA (INdAM -- Istituto Nazionale di Alta Matematica) Research Project 2020 and by the University of Milan, Research Support Plan PSR 2020.

\end{document}